\documentclass[11pt,english,leqno]{article}
\usepackage[T1]{fontenc}
\usepackage{times}
\usepackage[latin9]{inputenc}
\usepackage{amsthm}
\usepackage{amsmath}
\usepackage{amssymb}
\usepackage{enumerate}
\usepackage{hyperref}

\makeatletter

\newtheorem{thm}{Theorem}
\newtheorem{lemma}[thm]{Lemma}

\newtheorem{cor}[thm]{Corollary}
\theoremstyle{definition}

\DeclareMathOperator{\rank}{rank}

\newsavebox{\fmbox}

\newcommand{\reals}{\ensuremath{\mathbb{R}}}

\newlength{\dyindent}
\newenvironment{dy}{\refstepcounter{equation}\begin{list}{}%
{\setlength{\leftmargin}{\dyindent}\setlength{\labelwidth}{\dyindent}%
\addtolength{\labelwidth}{-\labelsep}}%
\item[{\rm (\theequation)\hfill}]}%
{\end{list}}

\newenvironment{dy*}{\refstepcounter{equation}\begin{list}{}%
{\setlength{\leftmargin}{\dyindent}\setlength{\labelwidth}{\dyindent}%
\addtolength{\labelwidth}{-\labelsep}}%
\item}%
{\end{list}}

\DeclareMathOperator{\mr}{mr}

\makeatother

\usepackage{babel}

\begin{document}
\global\long\def\lspan{\mbox{span}}
\global\long\def\I{\mbox{}{I}}

\title{The inverse inertia problem for the complements of partial $k$-trees}

\author{Hein van der Holst  \\
Department of Mathematics and Statistics \\
Georgia State University \\
Atlanta, GA 30303, USA\\
E-mail: hvanderholst@gsu.edu
}
\maketitle
\begin{abstract}
Let $\mathbb{F}$ be an infinite field with characteristic different from two. 
For a graph $G=(V,E)$ with $V=\{1,\ldots,n\}$, let $S(G;\mathbb{F})$ be the
set of all symmetric $n\times n$ matrices $A=[a_{i,j}]$ over $\mathbb{F}$ with $a_{i,j}\not=0$,
$i\not=j$ if and only if $ij\in E$. We show that if $G$ is the
complement of a partial $k$-tree and $m\geq k+2$, then for all nonsingular symmetric $m\times m$ matrices $K$ over $\mathbb{F}$, there exists an $m\times n$ matrix $U$ such that $U^T K U\in S(G;\mathbb{F})$. As a corollary we obtain that, if $k+2\leq m\leq n$ and $G$ is the complement of a partial $k$-tree, then for any two nonnegative integers $p$ and $q$ with $p+q=m$, there exists a matrix in $S(G;\reals)$ with 
$p$ positive and $q$ negative eigenvalues. 
\end{abstract}

\noindent keywords: graph, inertia, symmetric, graph complement, treewidth\newline

\section{Introduction}
Let $\mathbb{F}$ be a field. For a graph $G=(V,E)$ with $V=\{1,\ldots,n\}$, let $S(G;\mathbb{F})$ be the set of all symmetric $n\times n$ matrices $A=[a_{i,j}]$ over $\mathbb{F}$ with $a_{i,j}\not=0$,
$i\not=j$ if and only if $ij\in E$. We write $S(G)$ for $S(G;\mathbb{R})$. Consider the following problem for a given graph $G$: For which symmetric matrices $A$ over $\mathbb{F}$ does there exists a matrix $U$ such that $U^T A U\in S(G; \mathbb{F})$? This problem includes the inverse inertia problem for graphs. The inverse inertia problem of graphs has been introduced and studied by Barrett, Hall, and Loewy \cite{BarHalLoe2008a}, and asks for which pairs $(p,q)$ of nonnegative integers, there exists a matrix $A\in S(G)$ with $p$ positive and $q$ negative eigenvalues. The inverse inertia problem of graphs includes the problems of determining the minimum rank and minimum semidefinite rank of graphs.

The \emph{minimum rank} of $G$ over a field $\mathbb{F}$, denoted $\mr(G; \mathbb{F})$, is defined as
$$\mr(G;\mathbb{F}) = \min\{\rank(A)~|~A\in S(G;\mathbb{F})\},$$
and the \emph{minimum semidefinite rank} of $G$, denoted $\mr_{+}(G)$, is defined
as 
$$\mr_{+}(G)=\min\{\rank(A)~|~A\in S(G),A\text{ is positive semidefinite}\}.$$
We write $\mr(G)$ for $\mr(G; \mathbb{R})$. Clearly, $\mr(G)\leq \mr_+(G)$. 
For many classes of graphs a combinatorial characterization of the minimum rank has been established. For example, Johnson and Leal Duarte \cite{JD99} showed that the minimum rank of a tree equals the minimum number of disjoint paths needed to cover all vertices of the tree. Barrett, Loewy, and van der Holst \cite{MR2111528, MR2202431} gave for any field $\mathbb{F}$ a combinatorial characterization of the class of graphs $G$ with $\mr(G;\mathbb{F})\leq 2$. Also for the class of complement of trees, the minimum rank has been determined \cite{AIMMinimumRankSpecialGraphsWorkGroup20081628}. Hogben \cite{Hogben20082560} determined the minimum rank of the complements of $2$-trees. Sinkovic and van der Holst \cite{HolSin2010a} showed that the minimum semidefinite rank of the complement of a partial $k$-trees is at most $k+2$ (see below the definition of partial $k$-tree). See Hogben and Fallat \cite{Fallat2007558} for a survey on the minimum rank problem. 

Before stating the main result in this paper, we need to introduce some notions.
If $G=(V,E)$ is a graph, the complement of $G$ is $\overline{G}=(V,\overline{E})$,
where $\overline{E}=\{vw~|~v,w\in V,v\not=w\text{ and }vw\not\in E\}$. If $S\subseteq V$, then the induced subgraph of $G$ induced by $S$ is the subgraph of $G$ with vertex set $S$ and edge set $\{ij\in E~|~i,j \in S\}$.
A \emph{$k$-tree} is defined recursively as follows.
\begin{enumerate}
\item A complete graph with $k+1$ vertices is a $k$-tree.
\item If $G=(V,E)$ is a $k$-tree and $v_{1},\ldots,v_{k}$ form a clique in $G$ with $k$ vertices, then $G'=(V\cup\{v\},E\cup\{v_{i}v~|~1\leq i\leq k\})$, with $v$ a new vertex, is a $k$-tree.
\end{enumerate}
A \emph{partial $k$-tree} is a subgraph of a $k$-tree. A graph has tree-width $\leq k$ if it is a partial $k$-tree. We refer to Bodlaender \cite{bodlaender98a} for a survey on tree-width and to Diestel \cite{Diestel} for notation and terminology used in graph theory.

In this paper we show that if $G$ is the complement of a partial $k$-tree and $K$ is a nonsingular symmetric $m\times m$ matrix $K$ over an infinite field $\mathbb{F}$ with characteristic unequal to two, and $m\geq k+2$, then there exists a matrix $U$ such that $U^T K U\in S(G; \mathbb{F})$. Furthermore, if $G$ has $n$ vertices, then $\rank(U) = \min(m,n)$. Our result extends the result of Sinkovic and van der Holst \cite{HolSin2010a}.

\section{Symmetric bilinear forms}

Let $V$ be a finite dimensional vector space over a field $\mathbb{F}$ with characteristic different from $2$. A \emph{symmetric bilinear form} on $V$ is a map $B: V\times V\to \mathbb{F}$ satisfying
\begin{enumerate}
\item $B(u,v) = B(v,u)$ for all $u,v\in V$,
\item $B(u+v,w) = B(u,w) + B(v,w)$ for all $u,v,w\in V$,
\item $B(\lambda u,v) = \lambda B(u,v)$ for all $\lambda\in \mathbb{F}$ and all $u,v\in V$.
\end{enumerate}
Let $\mathcal{E} = \{x_1,\ldots,x_n\}$ be a basis for $V$. Define the matrix $K=[k_{i,j}]$ by $k_{i,j} = B(x_i,x_j)$.
Denoting the coordinate vector of a vector $x$ relative to $\mathcal{E}$ by $[x]_{\mathcal{E}}$, we have $B(u,v) = [u]_{\mathcal{E}}^T K [v]_{\mathcal{E}}$. Conversely, if $K$ is a symmetric $n\times n$ matrix over $\mathbb{F}$, then the map $B : V\times V\to \mathbb{F}$ defined by $B(u,v) = [u]_{\mathcal{E}}^T K [v]_{\mathcal{E}}$ is a symmetric bilinear form on $V$.

A \emph{B-orthogonal representation} of a graph $G=(V,E)$ in $V$ is a mapping $v\rightarrow\overrightarrow{v}$, $v\in V$, such that for distinct vertices $v$ and $w$,
$B(\overrightarrow{v},\overrightarrow{w}) = 0$ if and only if $v$ and $w$ are non-adjacent. In case $B$ is the standard inner product on $\mathbb{R}^k$, it is easy to verify that a graph $G$ has a $B$-orthogonal representation if and only if $\mr_+(G)\leq k$. 

A symmetric bilinear form $B$ on $V$ is \emph{nondegenerate} if $B(v,w) = 0$ for all $w\in V$ implies that $v=0$. If $\{x_1,\ldots,x_n\}$ is a set of vectors of $V$, then we say that $\{x_1,\ldots,x_n\}$ is \emph{nondegenerate} if
\begin{equation*}
\det\begin{bmatrix}
B(x_1,x_1) & \cdots & B(x_1,x_n)\\
\vdots & \ddots & \vdots\\
B(x_n,x_1) & \cdots & B(x_n,x_n)
\end{bmatrix}\not= 0.
\end{equation*}
If $\{x_1,\ldots,x_n\}$ is a basis for $V$, then $B$ is nondegenerate if and only if $\{x_1,\ldots,x_n\}$ is nondegenerate.
If $W$ is a subspace of $V$, we say that $W$ is nondegenerate if the restriction of $B$ to $W$ is nondegenerate.

Let $B$ be a symmetric bilinear form on $V$. 
If $W$ is a subspace of $V$, we define the \emph{orthogonal complement} of $W$ by
\begin{equation*}
W^{\perp} = \{x\in V~|~B(x,y)=0\text{ }\forall  y \in W\}.
\end{equation*}
A vector $v\in V$ is called \emph{anisotropic} if $B(v,v)\not=0$ and \emph{isotropic} if $v\not=0$ and $B(v,v)=0$. 

The following three lemmas will be used in the proof of the main theorem. We skip the proofs.

\begin{lemma}
If $B$ is a nondegenerate symmetric bilinear form on $V$, then for any linear subspace $W$ of $V$,
\begin{equation*}
\dim W + \dim W^\perp = \dim V,
\end{equation*}
and 
\begin{equation*}
(W^{\perp})^{\perp} = W.
\end{equation*}
\end{lemma}

\begin{lemma}
Let $B$ be a symmetric bilinear form on $V$. If $W$ is a nondegenerate subspace $W$ of $V$, then $V = W\oplus W^\perp$.
\end{lemma}

\begin{lemma}\label{lem:nondegenerate}
Let $B$ be a nondegenerate symmetric bilinear form on $V$, and let $W$ be a subspace of $V$. Then $W$ is nondegenerate if and only if $W^\perp$ is nondegenerate.
\end{lemma}

The next lemma allows us to reduce the number of cases in the proof of the main theorem.

\begin{lemma}\label{lem:orth}
Let $B$ be a nondegenerate symmetric bilinear form on $V$.
Let $K,L$ be linear subspaces of $V$ with $\dim K=\dim L$.
Then $\dim(L\cap K^{\perp})=\dim(L^{\perp}\cap K)$. \end{lemma}
\begin{proof}
We have \[
\dim V =\dim L+\dim K^{\perp}=\dim(L\cap K^{\perp})+\dim(L+K^{\perp}),\]
 and $(L^{\perp}\cap K)^{\perp}=L+K^{\perp}$. Hence $\dim(L\cap K^{\perp})=\dim(L^{\perp}\cap K)$.
\end{proof}

\section{The proof of the main theorem}

In the proof of our main theorem, we need a generic position argument. For this, we will use Lemma~\ref{lem:polysubspacenonzero}. In the proof of Lemma~\ref{lem:polysubspacenonzero}, we will use the following lemma; see \cite{MR1417938} for a proof of this lemma.
\begin{lemma}\label{lem:polynonzero}
Let $P(x_1,\ldots,x_n)$ be a nonzero polynomial over an infinite field  $\mathbb{F}$. Then there exist elements $a_1,\ldots,a_n \in \mathbb{F}$ such that $P(a_1,\ldots,a_n)\not=0$. 
\end{lemma}

If $L$ is a subspace of $\mathbb{F}^n$ and $P(x_1,\ldots,x_n)$ is a polynomial, then we say that $P(x_1,\ldots,x_n)$ is \emph{nonzero on $L$} if there exists a vector $a\in L$ such that $P(a_1,\ldots,a_n)\not=0$.

\begin{lemma}\label{lem:polysubspacenonzero}
Let $\mathbb{F}$ be an infinite field and $L$ a subspace of $\mathbb{F}^n$. If the polynomials $P_1(x_1,\ldots,x_n),\ldots,P_k(x_1,\ldots,x_n)$ are nonzero on $L$, then there exists a vector $a\in L$ such that $P_i(a_1,\ldots,a_n)\not=0$ for $i=1,\ldots,k$.
\end{lemma}
\begin{proof}
Let $r = \dim(L)$ and $B$ an $n\times r$ matrix with full column rank such that the column space of $B$ is equal to $L$. Since each $P_i(x_1,\ldots,x_n)$ is nonzero on $L$, there exists a vector $u\in \mathbb{F}^r$ such that $P_i(B u)$ is nonzero. Define $Q_i(y_1,\ldots,y_r) = P^i(B y)$ for $i=1,\ldots,k$. Then $Q_i(y_1,\ldots,y_r)$, $i=1,\ldots,k$, are nonzero polynomials. Let $Q(y_1,\ldots,y_r) = \prod_{i=1}^k Q_i(y_1,\ldots,y_r)$. Then, by Lemma~\ref{lem:polynonzero}, there exists a vector $c \in \mathbb{F}^r$ such that $Q(c) \not= 0$. Let $a = B c$. Then $a\in L$ and $P_i(a)\not=0$ for $i=1,\ldots,k$.
\end{proof}

We have now come to our main theorem. 

\begin{thm}\label{thm:orth1}
Let $\mathbb{F}$ be an infinite field and let $k$ and $m$ be positive integers with $k+2\leq m$. If $G=(V,E)$ is a partial $k$-tree, then for any nondegenerate symmetric bilinear form $B$ on $\mathbb{F}^{m}$, there is a mapping $v\rightarrow\overrightarrow{v}$, $v\in V$, into $\mathbb{F}^{m}$ such that 
\begin{enumerate}
\item $B(\overrightarrow{v},\overrightarrow{w})=0$ if and only if $v\not=w$ and $v$ is adjacent to $w$.
\item $\lspan(\{\overrightarrow{v}~\mid~v\in V(G)\})$ is a subspace of $\mathbb{F}^m$ with dimension $\min(m, |V(G)|)$.
\end{enumerate}
\end{thm}

\begin{proof}
In the proof, for any $S\subseteq V(G)$ we abbreviate $\lspan(\{\overrightarrow{v}~\mid~v\in S\})$ by $\lspan(S)$.
We prove a stronger statement: for every $k$-tree $H$ and every spanning subgraph $G$ of $H$, there is a mapping $v\rightarrow\overrightarrow{v}$, $v\in V(H)$, into $\mathbb{F}^{m}$ that satisfies the following 
conditions:
\begin{enumerate}[(1)]
\item \label{item:cond1} for all vertices $v,w$ of $G$, $B(\overrightarrow{v}, \overrightarrow{w}) = 0$ if and only if $v\not=w$ and $v$ is adjacent to $w$ in $G$.

\item\label{item:cond2} For every clique $C$ of $H$, $\lspan(C)$ is nondegenerate with dimension $|C|$. 

\item \label{item:cond3} For every pair of $k$-cliques $C,D$ of $H$, $\lspan(C)^{\perp}\cap\lspan(D)$ has dimension at most one. 

\item \label{item:cond4} For every $k$-clique $C$ of $H$ and vertex $v$ of $H$ with $v\not\in C$,
$\overrightarrow{v}\not\in\lspan(C)$.

\item\label{item:cond5} $\lspan(V(H))$ is a subspace of $\mathbb{F}^m$ with dimension $\min(m, |V(H)|)$.
\end{enumerate}

Notice that from Lemma~\ref{lem:orth} it follows that Condition \ref{item:cond3} is a condition on unordered pairs $C,D$ only.
The proof is by induction on the number of vertices in $H$. The basis of the induction is where $H=K_{k+1}$. It is easily verified that in this case there exists a mapping $v\to\overrightarrow{v}$, $v\in V(H)$,
that satisfies the Conditions \ref{item:cond1} - \ref{item:cond5}.

Now assume that the theorem is true for all $k$-trees $H$ with at most $n$ vertices and all subgraphs $G$ of $H$.
Let $H'$ be a $k$-tree with $n+1$ vertices and let $G'$ be a subgraph of $H'$. 
Let $z$ be a vertex of degree $k$ in $H'$ and let $Q=\{v_1,\ldots,v_k\}$ be the set of vertices in $H'$ adjacent to $z$. For $i=1,\ldots,k$, define $Q_{i}=(Q\setminus\{v_{i}\})\cup \{z\}$.

By induction there exists a mapping $v\rightarrow\overrightarrow{v}$, $v\in V(H)$, into $\mathbb{F}^{m}$ that satisfies the Conditions \ref{item:cond1} - \ref{item:cond5} for $H'\setminus \{z\}$ and $G'\setminus \{z\}$. We will show that we can extend this mapping to a mapping $v\rightarrow \overrightarrow{v}$, $v\in V(H')$, such that the Conditions \ref{item:cond1} - \ref{item:cond5} hold for $H'$ and $G'$. For this we need to assign to vertex $z$ a vector $\overrightarrow{z}\in L:=\lspan(N_{G'}(z))^{\perp}$ such that the following hold:
\begin{enumerate}[(a)]
\item For each clique $C$ of $H'$ containing $z$, $\lspan(C)$ is  nondegenerate with dimension $|C|$.
\item For each $Q_i$ and each $k$-clique $D$ of $H$, $\lspan(Q_i)^\perp\cap \lspan(D)$ has dimension at most one.
\item For each $Q_i$ and $Q_j$, $\lspan(Q_i)^\perp\cap \lspan(Q_j)$ has dimension at most one.
\item For each $k$-clique $C$ of $H$, $\overrightarrow{z}\not\in\lspan(C)$.
\item For each $Q_i$ and vertex $w\not\in Q_i$, $\overrightarrow{w}\not\in \lspan(Q_i)$.
\item If $|V(H)|<m$, then $\overrightarrow{z}\not\in \lspan(V(H))$.
\end{enumerate}

Clearly $\lspan(Q)^\perp\subseteq L$. Since $\lspan(Q)$ has dimension $k$, $\overrightarrow{v_i}\not\in \lspan(N_{G'}(z))$ if $v_i$ is nonadjacent to $z$. Hence $L\not\subseteq \lspan(v_i)^{\perp}$ if $v_i$ is nonadjacent to $z$. For each vertex $w$ of $G$ with $w\not\in Q$, 
$\overrightarrow{w}\not\in \lspan(N_{G'}(z))$ for otherwise $\overrightarrow{w}\in \lspan(Q)$, contradicting Condition~\ref{item:cond4}. Hence for each vertex $w$ of $G$ with $w\not\in Q$, $L\not\subseteq \lspan(w)^\perp$.
Hence there exists a nonzero polynomial $P^1(x)$ such that if $P^1(\overrightarrow{z})\not=0$, then $\overrightarrow{z}\not\in \lspan(w)^{\perp}$ for all $w\not\in N_{G'}(z)$, $w\not=z$.

Next we show that
\begin{dy}\label{dy:cliquenondegenerate}
For each clique $C$ of $H'$ containing $z$, there exists a nonzero polynomial $P^2_C(x)$ on $L$ such that if $P^2_C(\overrightarrow{z})\not=0$, then $\lspan(C)$ is a nondegenerate subspace of $\mathbb{F}^{m}$ with dimension $|C|$.
\end{dy}
Let $C=\{w_1,\ldots,w_t,z\}$ be a clique of $H'$ containing $z$.  Define 
\begin{equation*}
P_C^2(x) = \det\begin{bmatrix}
B(\overrightarrow{w_1},\overrightarrow{w_1}) & \cdots & B(\overrightarrow{w_1},\overrightarrow{w_t}) & B(\overrightarrow{w_1},x)\\
\vdots & \ddots & \vdots & \vdots\\
B(\overrightarrow{w_t},\overrightarrow{w_1}) & \cdots & B(\overrightarrow{w_t},\overrightarrow{w_t}) & B(\overrightarrow{w_t},x)\\
B(x,\overrightarrow{w_1}) & \cdots & B(x,\overrightarrow{w_t}) & B(x,x)
\end{bmatrix}.
\end{equation*}
Notice that $P^2_C(x)$ is a polynomial in the components of the vector $x$. The subspace $\lspan(C)$ is nondegenerate if and only if $P^2_C(\overrightarrow{z})\not=0$. To see that $P^2_C(x)$ is nonzero polynomial on $L$, first notice that $\lspan(Q)^\perp$ is nondegenerate as $\lspan(Q)$ is nondegenerate. Hence there exists an anisotropic vector $u \in \lspan(Q)^\perp$. Then $B(\overrightarrow{w_i},u) = 0$ for $i=1,\ldots,t$ and $B(u,u)\not=0$. Thus
\begin{equation*}
P_C^2(u) = \det\begin{bmatrix}
B(\overrightarrow{w_1},\overrightarrow{w_1}) & \cdots & B(\overrightarrow{w_1},\overrightarrow{w_t}) & 0\\
\vdots & \ddots & \vdots & \vdots\\
B(\overrightarrow{w_t},\overrightarrow{w_1}) & \cdots & B(\overrightarrow{w_t},\overrightarrow{w_t}) & 0\\
0 & \cdots & 0 & B(u,u)
\end{bmatrix},
\end{equation*}
which is nonzero on $\lspan(Q)^\perp$ because $\lspan(C\setminus\{z\})$ is nondegenerate. Hence $P^2_C(x)$ is a nonzero polynomial on $\lspan(Q)^\perp$. As $\lspan(Q)^\perp$ is a subspace of $L$, $P^2_C(x)$ is a nonzero polynomial on $L$.

Next we show that:
\begin{dy*}
For each $Q_i$ and each $k$-clique $D$ of $H$, there exists a nonzero polynomial $P_{Q_i,D}^3(x)$ such that if $P_{Q_i,D}^3(\overrightarrow{z})\not=0$, then $\lspan(Q_i)^\perp\cap \lspan(D)$ has dimension at most one.
\end{dy*}

Notice that, by Condition~\ref{item:cond3}, $\lspan(Q)^\perp\cap \lspan(D)$ has dimension at most one. Suppose first $\lspan(Q)^\perp\cap \lspan(D)$ contains a nonzero vector $h$.
Then $\lspan(Q)^\perp\not\subseteq \lspan(h)^\perp$, for otherwise $h\in \lspan(Q)$, and because $h\in \lspan(Q)^\perp$, $\lspan(Q)$ would be degenerate, contradicting Condition~\ref{item:cond2}. Thus $L\not\subseteq \lspan(h)^\perp$. Let $P_{Q_i,D}^3(x)$ be a nonzero polynomial on $L$ such that if $P_{Q_i,D}^3(\overrightarrow{z})\not=0$, then $\overrightarrow{z}\not\in \lspan(h)^\perp$. If $\overrightarrow{z}\not\in \lspan(h)^\perp$, then $h\not\in \lspan(\overrightarrow{z})^\perp$, and hence  $\lspan(Q)^{\perp}\cap \lspan(D)\cap\lspan(\overrightarrow{z})^{\perp}=\{0\}$, because the only vectors which $\lspan(Q)^\perp$ and $\lspan(D)$ have in common are scalar multiples of $h$. Suppose next that $\lspan(Q)^\perp\cap \lspan(D)$ contains no nonzero vectors. Then clearly $\lspan(Q)^\perp \cap \lspan(D)\cap \lspan(\overrightarrow{z})^\perp=\{0\}$. Hence, if $P_{Q_i,D}^3(\overrightarrow{z})\not=0$, then $\lspan(Q)^{\perp}\cap \lspan(D)\cap \lspan(\overrightarrow{z})^{\perp}=\{0\}$. From $\lspan(Q)^{\perp}\cap \lspan(D)\cap \lspan(\overrightarrow{z})^{\perp}=\{0\}$ it follows that $\lspan(Q_i)^\perp\cap \lspan(D)$ has dimension at most one.

Next we show that
\begin{dy}
If $P^2_{Q\cup\{z\}}(\overrightarrow{z})\not=0$, then for each $Q_i$ and $Q_j$, $\lspan (Q_i)^\perp\cap \lspan(Q_j)$ has dimension at most one.
\end{dy}
If $P^2_{Q\cup\{z\}}(\overrightarrow{z})\not=0$, then $\lspan(Q\cup\{z\})$ is nondegenerate, and hence $\lspan(Q\cup\{z\})^\perp\cap \lspan(Q\cup\{z\}) = \{0\}$. From this it follows that $\lspan(Q\cup\{z\})^\perp\cap\lspan(Q_j)=\{0\}$, and hence $\lspan(Q_i)^\perp\cap\lspan(Q_j)$ has dimension at most one.

Next we show that:
\begin{dy*}
For each $k$-clique $C$ of $H$, there exists a nonzero polynomial $P^4_{C}(x)$ such that if $P^4_{C}(\overrightarrow{z})\not=0$, then $\overrightarrow{z}\not\in \lspan(C)$.
\end{dy*}
To see this, notice that, by Condition~\ref{item:cond3}, $\lspan(Q)^{\perp}\cap \lspan(C)$ is a proper subspace of $\lspan(Q)^{\perp}$, and so $L\cap \lspan(C)$ is a proper subspace of $L$. Hence there exists a nonzero polynomial $P^4_{C}(x)$ such that $P^4_{C}(\overrightarrow{z})\not=0$ if and only if $\overrightarrow{z}\not\in \lspan(C)$.

We next show that:
\begin{dy*}
For each $Q_i$ and vertex $w\not\in Q_i$, there exists a nonzero polynomial $P^4_{Q_i,w}(x)$ such that if $P^4_{Q_i,w}(\overrightarrow{z})\not=0$, then $\overrightarrow{w}\not\in\lspan(Q_i)$.
\end{dy*}
First we show that, under the condition that $P^4_C(\overrightarrow{z})\not=0$ for all $k$-cliques $C$ of $H$, $\overrightarrow{w}\not\in\lspan(Q_i)$ is equivalent to:
\begin{equation*}
\overrightarrow{z}\not\in \lspan((Q_i\setminus\{z\})\cup\{w\}).
\end{equation*}
To see this, suppose $\overrightarrow{z}\in \lspan((Q_i\setminus\{z\})\cup\{w\})$. Then there are scalars $a_j$, $j\not=i$, and $b$ such that $\overrightarrow{z} = \sum_{j\not=i} a_j \overrightarrow{v_j} + b \overrightarrow{w}$. If $b=0$, then $\overrightarrow{z}\in \lspan(Q_i\setminus\{z\})\subseteq \lspan(Q)$.
However, $\overrightarrow{z}\not\in \lspan(C)$ for all $k$-cliques $C$ of $H$, so $b\not=0$. This implies that $\overrightarrow{w}\in \lspan(Q_i)$. For the converse implication, let $\overrightarrow{w}\in \lspan(Q_i)$. Then there are scalars $a_j$, $j\not=i$, and $b$ such that $\overrightarrow{w}=\sum_{j\not=i} a_j \overrightarrow{v_j} + b\overrightarrow{z}$. If $b=0$, then $\overrightarrow{w}\in \lspan(Q_i\setminus \{z\})$. If $w\not=v_i$, then $\overrightarrow{w}\in \lspan(Q)$, which contradicts Condition~\ref{item:cond4}. If $w=v_i$, then $\dim(\lspan(Q))<k$, which contradicts Condition~\ref{item:cond2} for $H$. Hence $b\not=0$, and so $z\in \lspan((Q_i\setminus\{z\})\cup \{w\})$.

To show that there exists a nonzero polynomial $P^4_{Q_i,w}(x)$ such that if $P^4_{Q_i,w}(\overrightarrow{z})\not=0$, then $\overrightarrow{w}\not\in\lspan(Q_i)$, we need to show that $L\cap  \lspan((Q_i\setminus\{z\})\cup\{w\})$ is a proper subspace of $L$. This follows from
\begin{dy}\label{dy:prop1}
$\lspan(Q)^{\perp}\cap \lspan((Q_i\setminus\{z\})\cup \{w\})$ has dimension at most one. 
\end{dy}
To prove this, we first show that $\lspan((Q_i\setminus\{z\})\cup \{w\})$ has dimension $k$.
If $w \in V(H)$ and $w\not\in Q$, then $w\not\in \lspan(Q)$. Since $\lspan(Q)$ has dimension $k$, $\lspan(Q\cup \{w\})$ has dimension $k+1$. Hence $\lspan((Q_i\setminus\{z\})\cup \{w\})$ has dimension $k$. If $w=v_i$, then $\lspan((Q_i\setminus\{z\})\cup \{w\}) = \lspan(Q)$, and hence $\lspan((Q_i\setminus\{z\})\cup \{w\})$ has dimension $k$. Let $u_1,\ldots,u_{k-1}$ be an orthogonal basis of $\lspan(Q_i\setminus\{z\})$. Since $\lspan(Q_i\setminus\{z\})$ is nondegenerate, the orthogonal complement of $\lspan(Q_i\setminus\{z\})$ in $\lspan(Q)$ is nondegenerate. Hence there exists a nonzero vector $u_k$ such that $u_1,\ldots,u_k$ is an orthogonal basis of $\lspan(Q)$. We are now ready to prove $(\ref{dy:prop1})$. Let $x\in \lspan(Q)$. There there scalars $\alpha_1,\alpha_2,\ldots,\alpha_k$ such that $x=\sum\alpha_i u_i$. Since $x\in \lspan(Q_i\setminus\{z\})^\perp$, $B(x,u_j)=0$ for $j=1,\ldots,k-1$. So $x=\alpha_k u_k$. Hence $\lspan(Q)\cap \lspan(Q_i\setminus\{z\})^\perp$ has dimension at most one. By 
Lemma~\ref{lem:orth}, the dimension of $\lspan(Q)^\perp\cap \lspan((Q_i\setminus\{z\})\cup \{w\})$ has the same dimension as $\lspan(Q)\cap \lspan((Q_i\setminus\{z\})\cup\{w\})^\perp = \lspan(Q)\cap\lspan(Q_i\setminus\{z\})^\perp\cap\lspan(w)^\perp$. Hence $(\ref{dy:prop1})$ follows.

Since $L\cap \lspan((Q_i\setminus\{z\})\cup \{w\})$ is a proper subspace of $L$, there exists a nonzero polynomial $P^4_{Q_i,w}(x)$ such that $P^4_{Q_i,w}(\overrightarrow{z})\not=0$ if and only if $\overrightarrow{z}\not\in \lspan((Q_i\setminus\{z\})\cup\{w\})$. From  $\overrightarrow{z}\not\in \lspan((Q_i\setminus\{z\})\cup\{w\})$ it follows that $\overrightarrow{w}\not\in \lspan(Q_i)$.

Next we show that
\begin{dy}
If $|V(H)|<m$, then there exists a nonzero polynomial $P^5(x)$ such that if $P^5(\overrightarrow{z})\not=0$, then $\overrightarrow{z}\not\in\lspan(V(H))$.
\end{dy}
Suppose $|V(H)|<m$. It is clear that $\lspan(Q)\subseteq \lspan(V(H))$. If $\lspan(Q)^\perp\subseteq \lspan(V(H))$, then $\mathbb{F}^m = \lspan(Q)\oplus \lspan(Q)^\perp\subseteq \lspan(V(H))$, as $\lspan(Q)$ is nondegenerate, which is a contradiction. Hence, $\lspan(Q)^\perp\cap \lspan(V(H))$ is a proper subspace of $\lspan(Q)^\perp$. From this it follows that $L\cap \lspan(V(H))$ is a proper subspace of $L$. Thus, there exists a nonzero polynomial $P^5(x)$ on $L$ such that if $P^5(\overrightarrow{z})\not=0$, then $\overrightarrow{z}\not\in\lspan(V(H))$.

By Lemma~\ref{lem:polysubspacenonzero}, there exists a vector $\overrightarrow{z}\in L$ such that $B(\overrightarrow{z},\overrightarrow{z})\not=0$, $P^1(\overrightarrow{z})\not=0$, $P^2_C(\overrightarrow{z})\not=0$, $P^3_{Q_i,D}(\overrightarrow{z})\not=0$, $P^3_{Q_i,Q_j}(\overrightarrow{z})\not=0$, $P^4_C(\overrightarrow{z})\not=0$, $P^4_{Q_i,w}(\overrightarrow{z})\not=0$, $P^5(\overrightarrow{z})$. Thus Conditions~\ref{item:cond1}-~\ref{item:cond5} hold for $H'$.
By induction the theorem holds for every $k$-tree.
\end{proof}

\begin{cor}\label{cor:orth}
Let $\mathbb{F}$ be an infinite field with characteristic unequal to two, and $k$ and $m$ positive integers such that  $m\geq k+2$. Let $K$ be a nonsingular symmetric $m\times m$ matrix over $\mathbb{F}$. If $G$ is the complement of a partial $k$-tree with $n$ vertices, then there exists an $m\times n$ matrix $U$ with $\rank(U) = \min(m, n)$ such that $U^T K U\in S(G;\mathbb{F})$.
\end{cor}

\begin{cor}
Let $\mathbb{F}$ be an infinite field with characteristic unequal to two and $k$ be a positive integer. If $G=(V,E)$ is the complement of partial $k$-tree, then $\mr(G; \mathbb{F})\leq k+2$.
\end{cor}

In the proof of Theorem~\ref{thm:pqeigenvalues}, we need the following lemma; its proof is standard and skipped.

\begin{lemma}\label{lem:pq}
Let $A$ be a real symmetric $m\times m$ matrix and let $U$ be a real  $m\times n$ matrix with $\rank(U)=m$. If $A$ has $p$ positive and $q$ negative eigenvalues, then $U^T A U$ has $p$ positive and $q$ negative eigenvalues.
\end{lemma}

\begin{thm}\label{thm:pqeigenvalues}
Let $G$ be a graph with $n$ vertices and let $k$ be a positive integer. If $k+2\leq m\leq n$ and $G$ is the complement of a partial $k$-tree, then for any two nonnegative integers $p$ and $q$ with $p+q=m$, there exists a matrix in $S(G;\reals)$ with 
$p$ positive and $q$ negative eigenvalues. 
\end{thm}
\begin{proof}
Let $D$ be a $m \times m$ diagonal matrix with $p$ and $q$ diagonal entries equal to $+1$ and $-1$, respectively.
By Corollary~\ref{cor:orth}, there exists an $m\times n$ matrix $U$ with $\rank(U) = m$ such that $A = U^T D U\in S(G)$. By Lemma~\ref{lem:pq}, $A$ has $p$ positive and $q$ negative eigenvalues. 
\end{proof}

\bibliographystyle{plain}
\bibliography{../../../biblio}

\end{document}